\documentclass[10pt]{amsart}
\usepackage{amsmath}
  \usepackage{graphics} 
\usepackage{graphicx}  
  \usepackage{epsfig} 
 \usepackage[colorlinks=true]{hyperref}
  \usepackage{float}
  \usepackage{caption}
\usepackage{amsmath}
\usepackage{wrapfig}
\usepackage{graphicx}
\usepackage{subcaption}
\usepackage{amssymb}

\newtheorem{theorem}{Theorem}[section]
\newtheorem{corollary}{Corollary}

\newtheorem{lemma}[theorem]{Lemma}

\theoremstyle{definition}

\newtheorem{remark}{Remark}

\title[ TYC blow-up ]
    {Large and Small Data Blow-Up Solutions in the Trojan Y Chromosome Model}

\begin{document}
\maketitle
\centerline{\scshape Rana D. Parshad$^{1}$, Matthew A. Beauregard$^{2}$, }
\centerline{\scshape  Eric M. Takyi$^{1}$, Thomas Griffin$^{1}$, Landrey Bobo$^{3}$ }
\medskip
{\footnotesize
 \centerline{1) Department of Mathematics,}
 \centerline{Iowa State University,}
   \centerline{Ames, IA 50011, USA.}
   \medskip
   \centerline{2) Department of Mathematics,}
 \centerline{Stephen F. Austin State University,}
  \centerline{Nacogdoches, TX 75962 , USA }
    \medskip
      \centerline{3) University of Texas,}
  \centerline{Austin, TX 78712 , USA }
    \medskip

}

\begin{abstract}
The Trojan Y Chromosome Strategy (TYC) is an extremely well investigated biological control method for controlling invasive populations with an XX-XY sex determinism. In \cite{GP12, WP14} various dynamical properties of the system are analyzed, including well posedness, boundedness of solutions, and conditions for extinction or recovery. These results are derived under the assumption of positive solutions.
In the current manuscript, we show that if the introduction rate of trojan fish is zero, under certain large data assumptions, negative solutions are possible for the male population, which in turn can lead to finite time blow-up in the female and male populations. A comparable result is established for \emph{any} positive initial condition if the introduction rate of trojan fish is large enough.  Similar finite time blow-up results are obtained in a spatial temporal TYC model that includes diffusion.  Lastly, we investigate improvements to the TYC modeling construct that may dampen the mechanisms to the blow-up phenomenon or remove the negativity of solutions. 
The results draw into suspect the reliability of current TYC models under certain situations. 
\end{abstract}

 \section{Introduction}

The detrimental effects of aquatic invasive species is well-documented \cite{A06, A12, B07, C01, L12, M00, 089, S97, V96}.  Current control methods rely primarily on chemical treatment \cite{SL15} and are environmentally detrimental.
The Trojan Y chromosome strategy (TYC) is a promising eradication strategy which circumvents such detriment \cite{GutierrezTeem06, SL15,TGP13}.  It involves introducing a YY male or YY male and feminised YY male into an invasive population. The off-spring of the YY male or YY feminised male are only wild type males or YY males. This skews the sex ratio of subsequent generations towards all males, and extinction of the population may occur (see Fig.~\ref{Fig:TYC-Eradication-Strategy}).

 \begin{figure}[H]
 \centering{
  {\includegraphics[height=4cm, width=11.5cm]{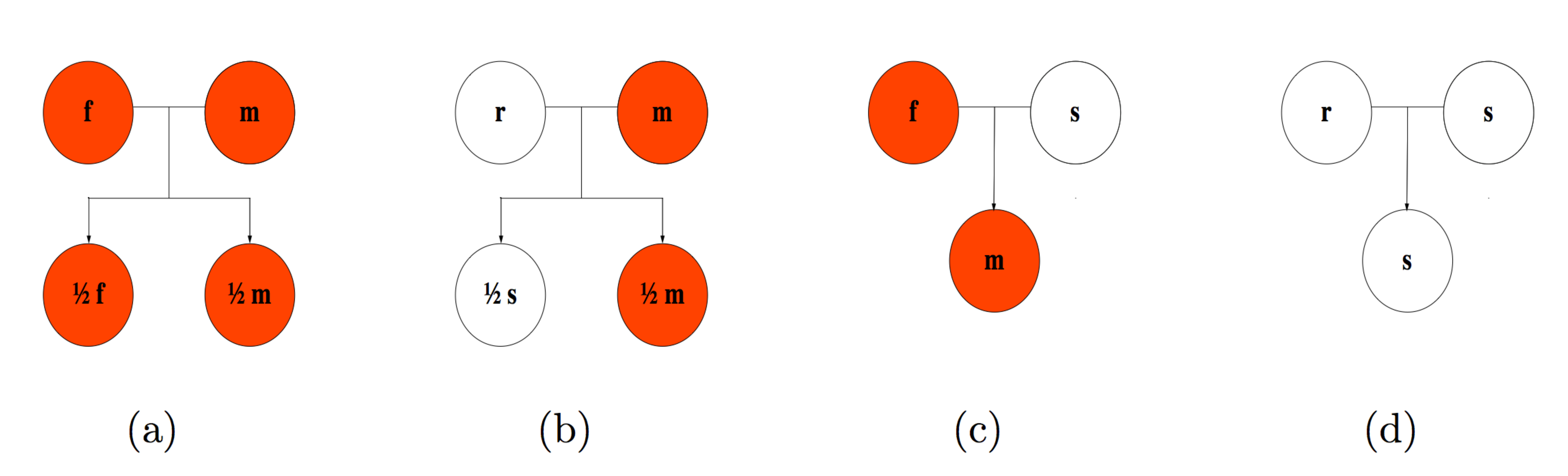}
  }
 \caption{\small The pedigree tree of the TYC model (that demonstrates {\it Trojan Y-Chromosome} eradication strategy).
 (a) Mating of a wild-type XX female (f) and a wild-type XY male (m). (b) Mating of a wild-type XY male (m) and a sex-reversed YY female (r). (c) Mating of a wild-type XX female (f) and a YY supermale (s). (d) Mating of a sex-reversed YY trojan female (r) and a YY supermale (s). Red color represents wild types, and white color represents phenotypes.}

 \label{Fig:TYC-Eradication-Strategy}
 }
 \end{figure}

The classical model of the TYC strategy was first pioneered in \cite{GutierrezTeem06}.  A detailed analysis and spatial effects were developed in \cite{Gutierrez12,ParshadGutierrez11} and is give below:
\begin{eqnarray}
\label{TYCOrgFeq} \partial_t f &=& D \Delta f + \frac{1}{2} f m \beta L  - \delta f, \\
\label{TYCOrgMeq} \partial_t m &=& D \Delta m + \left(\frac{1}{2}fm+\frac{1}{2}rm+fs \right) \beta L  -\delta m,  \\
\label{TYCOrgSeq} \partial_t s  &=& D \Delta s + \left(\frac{1}{2}rm+rs\right)\beta L  - \delta s, \\
\label{TYCOrgReq} \partial_t r  &=& D \Delta r + \mu  - \delta r,
\end{eqnarray}
with appropriate initial and boundary conditions that might be Dirichlet, Neumann or mixed. The functions $f,~m,~s,$ and $r$ are the population densities of individuals in each associated class: XX females, XY males, YY males, and YY females, respectively. The logistic term, $L = (1 - (f+m+s)/K)$, where $K$ is the carrying capacity of the ecosystem, attempts to penalize or encourage growth of populations when above or below $K$, respectively. The positive constants $\beta$ and $\delta$ represent the per capita birth and death rates, respectively; nonnegative constant $\mu$ denotes the rate at which the YY females $r$ are introduced; The species move by diffusion, with equal diffusivity coefficients $D$.  The TYC strategy is now in current practice and field studies that investigate the survivability and reproductively or introduced have been reported \cite{Schill2017}.  In recent years, the model has been the subject of intense theoretical investigation \cite{CW07a, CW07b, CW09, Gutierrez12, Schill18, P11, ParshadGutierrez11, p10, Parshad13, P09,Schill16,TGP13,SDE2013, WP14, Z12}.

It is clear that an equilibrium solution to TYC model (when $\mu =0$) is the trivial solution $(0,0,0,0)$, which is called the extinction state.  In the absence of super males or females and if $\frac{\delta}{\beta} < \frac{K}{16}$ then the only equilibrium solution is the extinction state, which is locally stable.  A key result is established in \cite{ParshadGutierrez11,WP14} and paraphrased here:
\begin{theorem}
Let $\frac{\delta}{\beta} < \frac{K}{16}$.  If $\mu=0$, the extinction state, $(0,0,0,0)$ is locally stable. In the presence of supermales/females the equilibrium solution, $(0,0,0,r^{*})$, where $r^* = \mu/\delta$ is locally stable if $\beta \mu^{2} - \beta K \delta \mu + K \delta^{3}  > 0$ and can be made globally stable for large enough $\mu$.
\end{theorem}

This theorem encapsulates an important theoretical result, that eradication is always possible, if the introduction of supermales, $\mu$, is sufficiently large.  This key result relies on the assumption that solutions remain positive for all initial nonnegative conditions. However, to the best of our knowledge, the positivity of solutions has not been proven rigorously. Subsequently, the analytical results are correct if solutions remain positive.

In this paper we show:

\begin{enumerate}
\item The spatially independent TYC system, \eqref{ClassicTYCfeq}-\eqref{ClassicTYCseq}, can have negative solutions, in the male population $m$, in certain situations. This result is given via Lemma \ref{lem:1a}.

\item The negativity of male population may lead to finite time blow-up in the female population towards $+\infty$ and the male population towards $- \infty$, for $\mu=0$ and large enough initial data. This is shown via Theorem \ref{thm:2a}.

\item For any positive initial data there exists a critical introduction rate $\mu^{*} > 0$, such that, $\forall \mu > \mu^{*}$ finite time blow-up occurs in \eqref{ClassicTYCfeq}-\eqref{ClassicTYCseq}. This is a consequence of Theorem \ref{thm:2a1}.

\item Similar results are derived in the spatial-temporal (PDE) model given by \eqref{Cpf}-\eqref{Cps}. This is shown via Theorems \ref{thm:3a} and \ref{thm:3apd}.

\end{enumerate}

We then discuss the practical relevance of these results to biological control and possible model corrections and restrictions.

\section{Finite Time Blow-Up in Dynamical Systems Model}

The classical population model of the TYC strategy relates the populations of the wild-type XX females ($f$), wild-type XY males ($m$), and the YY supermale ($s$) populations over time. However, in current field experiments only the supermale population has been introduced into the wild \cite{Schill18}.  Subsequently, in this section we investigate the three species TYC model, that is,
\begin{eqnarray}
\label{ClassicTYCfeq} \dot{f} &=& F(f,m,s) = \frac{1}{2}  \beta L fm - \delta f ,\\
\label{ClassicTYCmeq} \dot{m} &=& G(f,m,s) = \frac{1}{2}  \beta L fm + \beta L fs- \delta m,\\
\label{ClassicTYCseq} \dot{s} &=& H(f,m,s) = \mu - \delta s,
\end{eqnarray}
where the logistic term $L$ and other parameters are as defined earlier.  Again, the parameters are assumed to nonnegative.  In addition, we assume positive initial conditions $(f_{0}, m_{0}, s_{0})$.

\subsection{Negative Solutions}
We recall a result guaranteeing nonnegativity of solutions from \cite{P00, P10}:

\begin{lemma}
\label{lem:l1}
Consider the following ODE system,
\begin{eqnarray}
\label{C1} \dot{f} &=&  F(f,m,s),\\
\label{C2} \dot{m} &=&  G(f,m,s),\\
\label{C3} \dot{s} &=&  H(f,m,s).
\end{eqnarray}
Nonnegativity of solutions is preserved in time, that is,
\begin{equation*}
f_{0},m_{0},s_{0} \geq 0~~ \Rightarrow~~ \forall t \in [0, T_{max}),~ f(t),m(t),s(t) \geq 0.
\end{equation*}
if and only if
\begin{equation*}
\forall f(t),m(t),s(t) \geq 0,
\end{equation*}
we have that
\begin{equation*}
F(0,m,s),~ G(f,0,s),~ H(f,m,0) \geq 0.
\end{equation*}
\end{lemma}

Notice that by \eqref{ClassicTYCmeq} that $G(f,0,s) = \beta f s \left( 1 - (f+s)/K\right)$.  Hence, if $f(t)+s(t)>K$ at some time $t$ then $1-(f+s)/K<0$ which gives $G(f,0,s)<0$, which violates the above lemma.  This suggests negative solutions are possible with the current model of TYC.  This is stated in the following lemma.

\begin{lemma}
\label{lem:1a}
Consider the TYC system given by \eqref{ClassicTYCfeq}-\eqref{ClassicTYCseq}.
Then there exists positive initial data $(f_{0}, m_{0}, s_{0})$, and a time interval $[T_{1}, T_{2}] \in (0,\infty)$, s.t for solutions emanating from this data, the male population $m(t)$ is negative on $[T_{1}, T_{2}]$.
\end{lemma}

\begin{proof}

Consider \eqref{ClassicTYCfeq}-\eqref{ClassicTYCseq} then clearly we have,
\begin{equation*}
F(0,m,s) = 0,~ H(f,m,0) = \mu \geq 0,
\end{equation*}
Thus, $f,s \geq 0$ for all time. However,
\begin{equation*}
G(f,0,s) = \beta fs \left( 1 - \frac{f+s}{K} \right),
\end{equation*}
and clearly for initial data $f_{0}, s_{0} \gg K$, $G(f,0,s) <0$, violating the necessary requirement for positive solutions from Lemma \ref{lem:l1}, and will yield negative solutions in the male population $m$.
\end{proof}

\subsection{Finite Time Blow Up of Solutions}

Interestingly, since it is possible to obtain negative solutions in the male population, this then could lead to the possibility of finite time blow up in the female population. Finite time blow-up has a rich history, we refer the interested reader to \cite{P10, P00, L12, PQ07}. In the forthcoming theorems we show that finite blow up is possible for large enough initial data or introduction rate of supermales. The former result is given in the following Theorem.

\begin{theorem}
\label{thm:2a}
Consider the TYC system given by \eqref{ClassicTYCfeq}-\eqref{ClassicTYCseq}, with $\mu=0$.
Then there exists positive initial data $(f_{0}, m_{0}, s_{0})$, such that solutions emanating from this data, will blow-up in finite time, that is
\begin{equation*}
\limsup_{t \rightarrow T^{*} < \infty} f \rightarrow + \infty
\end{equation*}
or
\begin{equation*}
\limsup_{t \rightarrow T^{**} < \infty} m \rightarrow - \infty
\end{equation*}
\end{theorem}

\begin{proof}

Consider the equation for $f$ expanded, namely,
\begin{equation*}
\dot{f}  =  \frac{\beta}{2}  fm - \frac{\beta}{2K} f^{2}m - \frac{\beta}{2K} fm^{2} - \frac{\beta}{2K} fms - \delta f.
\end{equation*}
Now, by Lemma \ref{lem:1a} we know that for large initial data $s_{0}$ and $f_{0}$, say $f^{*}_{0},~s^{*}_{0} \gg K$, that $m(t) < 0$ for $t\in [T_{1}, T_{2}]$. Let
$m=-\tilde{m}$, where $\tilde{m}(t) > 0$ on $t\in [T_{1}, T_{2}]$, to yield,
\begin{equation*}
\dot{f}  =  -\frac{\beta}{2}  f \tilde{m} + \frac{\beta}{2K} f^{2}\tilde{m} - \frac{\beta}{2K} f \tilde{m}^{2} + \frac{\beta}{2K} f \tilde{m}s - \delta f.
\end{equation*}
Rearranging results in,
\begin{equation*}
\dot{f}  =   \frac{\beta}{2K} f^{2}\tilde{m} + \frac{\beta}{2K} f\tilde{m}s - \frac{\beta}{2}  f \tilde{m} - \frac{\beta}{2K} f \tilde{m}^{2} - \delta f.
\end{equation*}
Since $m < 0$ on $[T_{1}, T_{2}]$, we must have $-\delta_{2} < m < -\delta_{1}$. The upper bound is obvious, but we also have a lower bound, else the finite time blow-up of $m$ and subsequently $f$ follows trivially. Thus via standard comparison, we have,
\begin{equation*}
\dot{f}  \geq    \dot{\tilde{f}} = \frac{\beta \delta_{1}}{2K} f^{2}  - \frac{\beta \delta_{2}}{2}  f  - \frac{\beta (\delta_{2})^{2}}{2K} f  - \delta f
\end{equation*}
Clearly a solution to $\tilde{f}$ will blow-up at a finite time $T^{*}$ for sufficiently large data. That is, if
\begin{equation}
\label{f0threshold}
\tilde{f}_{0}  > \dfrac{ \dfrac{\beta \delta_{2}}{2}  + \dfrac{\beta (\delta_{2})^{2}}{2K} + \delta}{ \dfrac{\beta \delta_{1}}{2K} }
\end{equation}
Now we choose the data large enough s.t

\begin{equation*}
f_{0}=\max(\tilde{f}_{0} , f^{*}_{0} ).
\end{equation*}
This ensures that $T^{*} < T_{2}$ and since $\tilde{f}$ is a sub-solution to \eqref{ClassicTYCfeq}, this implies that $f$ blows-up in finite time. This completes the proof of the blow-up of $f$.

The method of proof for the blow-up of $m$ is similar. Essentially, by Lemma  \ref{lem:1a} we know that for large $s_{0}$ and $f_{0}$, that is say $f^{*}_{0},~s^{*}_{0} \gg K$, $m(t)<0$ on $[T_{1}, T_{2}]$. Thus if we consider the equation for $m(t)$ on $[T_{1}, T_{2}]$, we see
 \begin{equation*}
\dot{m}  =  \frac{\beta}{2}  fm + \beta fs  - \frac{\beta}{2K} f^{2}m - \frac{\beta}{2K} fm^{2} - \frac{\beta}{2K} fms  - \frac{\beta}{K} f s^{2} - \frac{\beta}{K} sf^{2} - \frac{\beta}{K} fms  - \delta m
\end{equation*}
We set $m=-\tilde{m}$, where $\tilde{m} > 0$ on $[T_{1}, T_{2}]$, to yield,
\begin{equation*}
\dot{\tilde{m}}  =  \frac{\beta}{2}  f\tilde{m} - \beta fs  - \frac{\beta}{2K} f^{2}\tilde{m} + \frac{\beta}{2K} f\tilde{m}^{2} - \frac{\beta}{2K} f\tilde{m}s  + \frac{\beta}{K} f s^{2} + \frac{\beta}{K} sf^{2} - \frac{\beta}{K} f\tilde{m}s - \delta \tilde{m}
\end{equation*}
We now assume $0 < \delta_{3} < f < \delta_{4}$, that is, we consider the case where $f$ has not blown up by some intermediate time.  The lower bound is obvious by positivity of $f$. The upper bound is guaranteed if blow up in $f$ has not occurred.  Clearly, if $f$ has blown up then $m(t)$ tends toward negative infinity in finite time.  We shall assume in the following that, indeed, $f$ has an upper bound and has not blown up.

A simple comparison using positivity now yields,
 \begin{eqnarray*}
\dot{\tilde{m}}  &\geq&  - \beta \delta_{4} s  - \frac{\beta}{2K} (\delta_{4})^{2}\tilde{m} + \frac{\beta \delta_{3}}{2K} \tilde{m}^{2} - \frac{3\beta}{2K} \delta_{4}\tilde{m}s    - \delta \tilde{m} \\
&\geq &  \frac{\beta \delta_{3}}{2K} \tilde{m}^{2}  - C_{1} \tilde{m},
\end{eqnarray*}
where $C_{1}$ is a constant that depends on the parameters in the problem. Next consider the differential equation
 \begin{equation*}
\dot{\tilde{m}}  =  \frac{\beta \delta_{3}}{2K} \tilde{m}^{2}  - C_{1} \tilde{m}
\end{equation*}
Then the solution to the above differential equation will have a solution that will blow-up at a finite time $T^{**}$ for sufficiently large data, that is, if
\begin{equation}
\label{m0threshold}
\tilde{m}_{0} > \frac{C_{1}}{ \frac{\beta \delta_{3}}{2K}}.
\end{equation}
By comparison, this implies that  $m$ will go to negative infinite in finite time. This completes the proof of the finite time blow-up of $m$.
\end{proof}

\begin{remark}
From the above proof it is not clear that the blow up time in the female and male population naturally coincide.  However, numerical results does show that this can occur.
\end{remark}

The previous theorem assumed that there was only an initial introduction of supermales through the initial condition $s(0)$.  By choosing large enough initial supermales it is clear now that the male population can be come negative and that the female or male population may blow up in finite time.  In the forthcoming theorem we turn our attention to the situation where $\mu\neq 0$, that is, the case of a constant introduction of supermales.  It will be shown that for any initial data that there exists a critical introduction rate that will lead to blow in finite time.

\begin{theorem}
\label{thm:2a1}
Consider the TYC system given by \eqref{ClassicTYCfeq}-\eqref{ClassicTYCseq}, with $\mu>0$.
For any positive initial data $(f_{0}, m_{0}, s_{0})$ large or small, there exists a critical $\mu^{*}(f_{0}, m_{0}, s_{0})$, such that for any $\mu > \mu^{*}(f_{0}, m_{0}, s_{0})$, solutions emanating from this data, will blow-up in finite time, that is
\begin{equation*}
\limsup_{t \rightarrow T^{*} < \infty} f \rightarrow + \infty
\end{equation*}
or
\begin{equation*}
\limsup_{t \rightarrow T^{**} < \infty} m \rightarrow - \infty
\end{equation*}
\end{theorem}

\begin{proof}

We first choose $\mu > \delta K$.  This guarantees that
\begin{equation*}
G(f,0,s) = fs\left(1-\frac{f+s}{K}\right) < 0
\end{equation*}
for any $f_{0}$, and so following the methods of Theorem \ref{thm:2a} we have negative solutions for $m$. Following the same estimates as in Theorem \ref{thm:2a} we arrive at:
\begin{equation*}
\dot{f}  =   \frac{\beta}{2K} f^{2}\tilde{m} + \frac{\beta}{2K} f\tilde{m}s - \frac{\beta}{2}  f \tilde{m} - \frac{\beta}{2K} f \tilde{m}^{2} - \delta f.
\end{equation*}
Now using the steady state solution of $s$, we have that
\begin{equation*}
\dot{f}  \geq   \frac{\beta}{2K} f^{2}\delta_{2} + \frac{\beta}{2K} f\delta_{2}\frac{\mu}{\delta} - \frac{\beta}{2}  f \tilde{m} - \frac{\beta}{2K} f \tilde{m}^{2} - \delta f.
\end{equation*}
This yields
\begin{equation*}
\dot{f}  \geq   \frac{\beta}{2K} f^{2}\delta_{2} + \frac{\beta}{2K} \delta_{2}\frac{\mu}{\delta} f - \left( \frac{\beta \delta_{2}}{2}  +\frac{\beta (\delta_{2})^{2} + \delta}{2K}  \right)f.
\end{equation*}
Now if
\begin{equation}
\label{muthreshold}
\mu^{*} = \frac{\left( \frac{\beta \delta_{2}}{2}  +\frac{\beta (\delta_{2})^{2} + \delta}{2K}  \right)}{\left( \frac{\beta \delta_{2}}{2K\delta}\right)},
\end{equation}
then for any $\mu > \mu^{*}$ we obtain
\begin{equation*}
\dot{f}  \geq  \left( \frac{\beta \delta_{2} }{2K} \right)f^{2}.
\end{equation*}
This fact leads to the finite time blow-up of $f$ for any positive initial condition. The blow-up of $m$ can be shown similarly.
\end{proof}

\begin{remark}
Note via \eqref{muthreshold}, $\mu^{*}$ depends on the parameters in the problem, as well as $\delta_{2}$ - which in turn depends on the initial conditions $(f_{0}, m_{0}, s_{0})$.
\end{remark}

\section{Finite Time Blow-Up in the Partial Differential Equations Model}
\label{PDE}
The disadvantageous behavior observed in the previous section may be attenuated in the presence of spatial pressures.  Here, we examine the spatially explicit version of the three species TYC model and prove the possibility of finite time blow-up in the female or male population. To this end we use standard techniques \cite{L12, PQ07, P00, P10} Consider the partial differential equations (PDE) model:
\begin{eqnarray}
\label{Cpf} \frac{\partial f}{\partial t} &=& \Delta f + \frac{1}{2}  \beta L fm - \delta f ,\\
\label{Cpm} \frac{\partial m}{\partial t}  &=& \Delta m + \frac{1}{2}  \beta L fm + \beta L fs- \delta m,\\
\label{Cps}\frac{\partial s}{\partial t} &=&\Delta s + \mu - \delta s ,
\end{eqnarray}
specified over the domain $(x,t) \subset \Omega \times (0,\infty)$ and subject to homogeneous Neumann boundary conditions on the boundary $\partial \Omega$, that is,
\begin{equation*}
 \nabla f \cdot n = \nabla m \cdot n = \nabla s \cdot n = 0,
\end{equation*}
where $L$, $\beta$, $\delta$, and $\mu$ are as defined previously.  Again, the parameters, $\beta$, $\delta$, and $\mu$ are assumed nonnegative. Positive initial data, $(f(x,0), m(x,0), s(x,0))$, is assumed herein.  We first show the possibility of finite time blow in the female or male populations in the situation for large enough initial data where $\mu = 0$.  We then show, regardless of the initial condition size, that there exists a critical introduction rate of supermales that leads to finite time blow-up.  Hence, even in the case of diffusion, the results of the previous section permeate into the PDE model.

\begin{theorem}
\label{thm:3a}
Consider the TYC system given by \eqref{Cpf}-\eqref{Cps}, with $\mu=0$.
Then there exists positive initial data $(f(x,0), m(x,0), s(x,0))$, such that solutions emanating from this data, can blow-up in finite time, that is
\begin{equation*}
\limsup_{t \rightarrow T^{*} < \infty} ||f||_{p} \rightarrow + \infty
\end{equation*}
and
\begin{equation*}
\limsup_{t \rightarrow T^{**} < \infty} ||m||_{p}  \rightarrow  \infty
\end{equation*}
for all $p\geq1$.

\end{theorem}

\begin{proof}

Consider the equation for the female population expanded:
\begin{equation*}
\frac{\partial f}{\partial t}   = \Delta f +  \frac{\beta}{2}  fm - \frac{\beta}{2K} f^{2}m - \frac{\beta}{2K} fm^{2} - \frac{\beta}{2K} fms - \delta f.
\end{equation*}
Via Lemma \ref{lem:1a} we know that for large $s(x,0)$ and $f(x,0)$, that is $s(x,0),f(x,0) \gg K$, that $m(x,t)<0$ for $t\in [T_{1}, T_{2}]$ and $x \in \Omega$. Let $m=-\tilde{m}$, where $\tilde{m} > 0$ for $t\in [T_{1}, T_{2}]$ and $x \in \Omega$. By direct substitution,
\begin{equation*}
\frac{\partial f}{\partial t}  = \Delta f  -\frac{\beta}{2}  f \tilde{m} + \frac{\beta}{2K} f^{2}\tilde{m} - \frac{\beta}{2K} f \tilde{m}^{2} + \frac{\beta}{2K} f \tilde{m}s - \delta f.
\end{equation*}
Integrating over the spatial domain $\Omega$ and rearranging yields,
\begin{eqnarray*}
\frac{d}{dt}\int_{\Omega} f\ dx  &=&   \frac{\beta}{2K}\int_{\Omega} f^{2}\tilde{m}\ dx + \frac{\beta}{2K}\int_{\Omega} f\tilde{m}s\ dx - \frac{\beta}{2} \int_{\Omega} f \tilde{m}\ dx  \\
& &- \frac{\beta}{2K} \int_{\Omega}f \tilde{m}^{2}\ dx- \delta \int_{\Omega}f\ dx.
\end{eqnarray*}
Since $m(x,t)<0$ for $(x,t) \in \Omega \times [T_1, T_2]$ then $-\delta_2 < m(x,t) < -\delta_1$, for positive constants $\delta_1$ and $\delta_2$.  Thus via standard comparison as earlier and H\"{o}lder's inequality we have,
\begin{equation*}
\frac{d}{dt}\int_{\Omega} f dx  \geq    \frac{\beta \delta_{1}}{2K}\left(\int_{\Omega} f dx\right)^{2} -  \left( \frac{\beta \delta_{2}}{2}  +\frac{\beta (\delta_{2})^{2} + \delta}{2K}  \right)\int_{\Omega} f dx,
\end{equation*}
Define $F(t)=\displaystyle \int_{\Omega} f dx$, then,
\begin{equation*}
\frac{d}{dt}F(t)  \geq    \frac{\beta \delta_{1}}{2K}\left(F(t)\right)^{2} -  \left( \frac{\beta \delta_{2}}{2}  +\frac{\beta (\delta_{2})^{2} + \delta}{2K}  \right)F(t)
\end{equation*}
which yields the finite time blow-up of $F(t)$, for large enough initial data. That is for

\begin{equation*}
F(0)  = \int_{\Omega} f(x,0)  dx  \geq   \frac{\left( \frac{\beta \delta_{2}}{2}  +\frac{\beta (\delta_{2})^{2} + \delta}{2K}  \right)}{\left( \frac{\beta \delta_{2}}{2K\delta}\right)}
\end{equation*}

Thus the $L^{1}(\Omega)$ norm of $f$ blows-up in finite time. Since $L^{p}(\Omega) \hookrightarrow L^{1}(\Omega)$, for $p \geq 1$, we have that the $L^{p}$ norm of $f$ blows up for any $p$, for large enough initial conditions.
This completes the proof of the blow-up of $f$.  A similar proof for the blow up in the male population can be established.
\end{proof}

The previous theorem proves that the finite time blow up is a possibility even if the only introduction of supermales is through the initial condition.  In the following theorem we prove that regardless of the initial condition size that there exists a threshold to the introduction rate, $\mu$, such that rates beyond this value will lead to finite time blow in the female population.

\begin{theorem}
\label{thm:3apd}
Consider the TYC system given by \eqref{Cpf}-\eqref{Cps}, with $\mu>0$.
Then for any positive initial data $(f(x,0), m(x,0), s(x,0))$, there exists a $\mu^{*}$ such that if $\mu > \mu^{*}$, then solutions emanating from this data, can blow-up in finite time, that is
\begin{equation*}
\limsup_{t \rightarrow T^{*} < \infty} ||f||_{p} \rightarrow + \infty
\end{equation*}
and
\begin{equation*}
\limsup_{t \rightarrow T^{**} < \infty} ||m||_{p}  \rightarrow \infty
\end{equation*}
for all $p\geq1$.
\end{theorem}

\begin{proof}

Let $\mu > \delta K$.  Following similar estimates as in the previous theorem yields,
\begin{equation*}
\frac{d}{dt}\int_{\Omega} f\ dx  =   \frac{\beta}{2K}\int_{\Omega} f^{2}\tilde{m}\ dx + \frac{\beta}{2K}\int_{\Omega} f\tilde{m}s dx - \frac{\beta}{2} \int_{\Omega} f \tilde{m}dx - \frac{\beta}{2K} \int_{\Omega}f \tilde{m}^{2}dx - \delta \int_{\Omega}fdx.
\end{equation*}
Via Holder's inequality we have,
\begin{equation*}
\frac{d}{dt}\int_{\Omega} f dx  \geq    C_{1}\frac{\beta \delta_{1}}{2K}\left(\int_{\Omega} f dx\right)^{2} +  \frac{\beta}{2K}\int_{\Omega} f\delta_{2}\frac{\mu}{\delta} dx -  \left( \frac{\beta \delta_{2}}{2}  +\frac{\beta (\delta_{2})^{2} + \delta}{2K}  \right)\int_{\Omega} f dx
\end{equation*}
The result follows for
\begin{equation}
\label{muthresholdPDE}
\mu > \dfrac{ \left( \dfrac{\beta \delta_{2}}{2}  +\dfrac{\beta (\delta_{2})^{2} + \delta}{2K}  \right)}{\dfrac{\beta \delta_{}2}{2K\delta}}
\end{equation}
\end{proof}

\begin{remark} The thresholds provided in the current and present section, namely  \eqref{f0threshold}, \eqref{m0threshold}, \eqref{muthreshold}, and \eqref{muthresholdPDE} are not guaranteed to be sharp.  In fact, numerical experiments suggest the critical values of initial condition size or introduction may be much smaller to yield negative solutions or finite time blow-up.
\end{remark}

This section's results motivate the following corollaries concerning the classical four species TYC model.

\begin{corollary}
Consider the TYC system given by \eqref{TYCOrgFeq}-\eqref{TYCOrgReq}, with $\mu=0$.
Then there exists positive initial data $(f(x,0), m(x,0), s(x,0), r(x,0))$, such that solutions emanating from this data, can blow-up in finite time, that is
\begin{equation*}
\limsup_{t \rightarrow T^{*} < \infty} ||f||_{p} \rightarrow + \infty
\end{equation*}
and
\begin{equation*}
\limsup_{t \rightarrow T^{**} < \infty} ||m||_{p}  \rightarrow  \infty
\end{equation*}
for all $p\geq1$.
\end{corollary}

\begin{corollary}
Consider the TYC system given by \eqref{TYCOrgFeq}-\eqref{TYCOrgReq}, with $\mu>0$.
Then for any positive initial data $(f(x,0), m(x,0), s(x,0), r(x,0))$, there exists a $\mu^{*}$ such that if $\mu > \mu^{*}$, then solutions emanating from this data, can blow-up in finite time, that is
\begin{equation*}
\limsup_{t \rightarrow T^{*} < \infty} ||f||_{p} \rightarrow + \infty
\end{equation*}
and
\begin{equation*}
\limsup_{t \rightarrow T^{**} < \infty} ||m||_{p}  \rightarrow  \infty
\end{equation*}
for all $p\geq1$.
\end{corollary}

%
%
%
%

\section{Numerical Experiments}

In this section, numerical experiments are given to illustrate that negativity of solutions or finite time blow up is a possibility with the classical three species model.  Comparable numerical results can be established for the four species situation.  In the numerical simulations, the dynamical system is scaled into dimensionless form.  In particular, we let $f \rightarrow \dfrac{f}{K}$, $m \rightarrow \dfrac{m}{K}$, $s \rightarrow \dfrac{s}{K}$ , $\tau \rightarrow \delta t$ and  $r=\dfrac{\beta K}{2\delta}$. The dimensionless variables $r$ is a ratio of the two time scales in the TYC model, that is, the birth and death rates.  The dimensionless system takes the form:
\begin{eqnarray}
\label{eq:nondimF} \dot{f} &=& rmfL -f \\
\label{eq:nondimM} \dot{m} &=& rmfL +2rsfL -m\\
\label{eq:nondimS} \dot{s} &=& \gamma-s
\end{eqnarray}
where the logistic term is $L=1-(f+m+s).$  In \cite{JingJing2019}, population experiments of guppy fish were given and subsequently used to determine the best parameters, in the least squares sense, to the mating model, \eqref{eq:nondimF}-\eqref{eq:nondimM}, with no supermales.  The best fit parameters suggested $r\approx 17.8125.$  This value of $r$ will be used throughout the numerical simulations. All simulations are computed using a Runge-Kutta-Fehlberg (RK45) and conducted in Matlab\circledR.

\begin{figure}[ht]
\centering
\begin{tabular}{cc}
\includegraphics[scale=0.16]{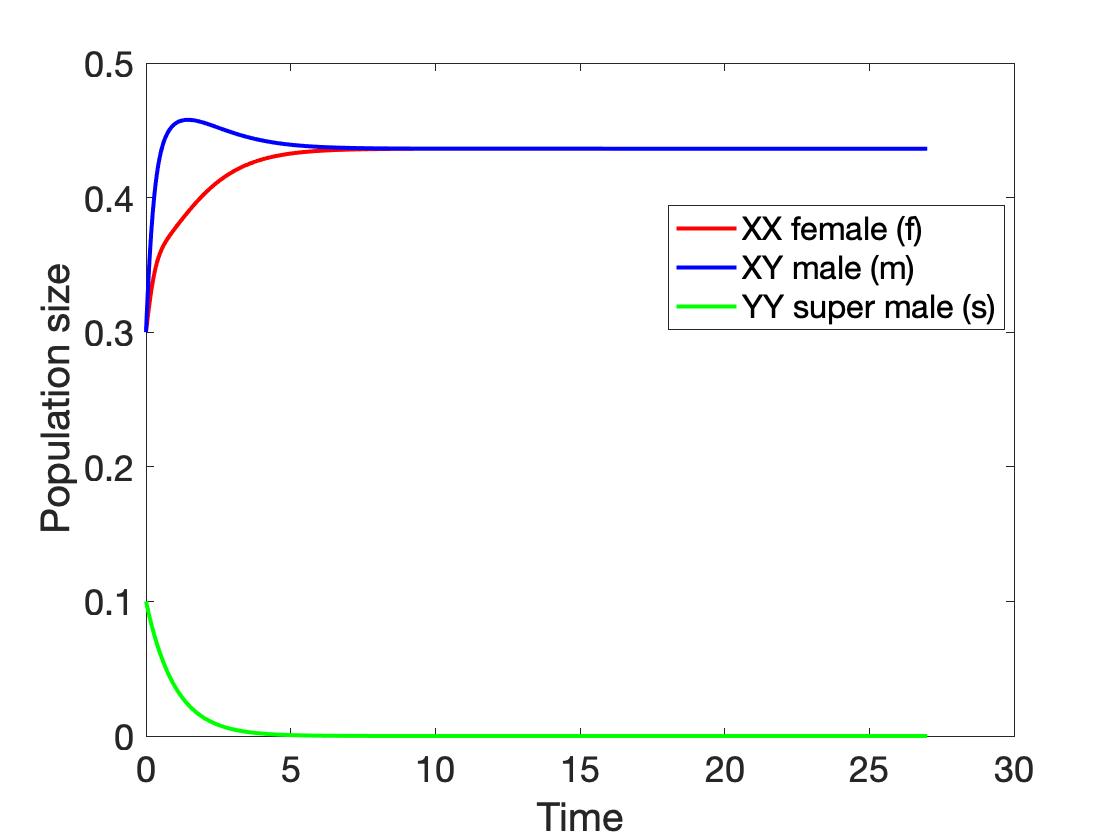} &
\includegraphics[scale=0.16]{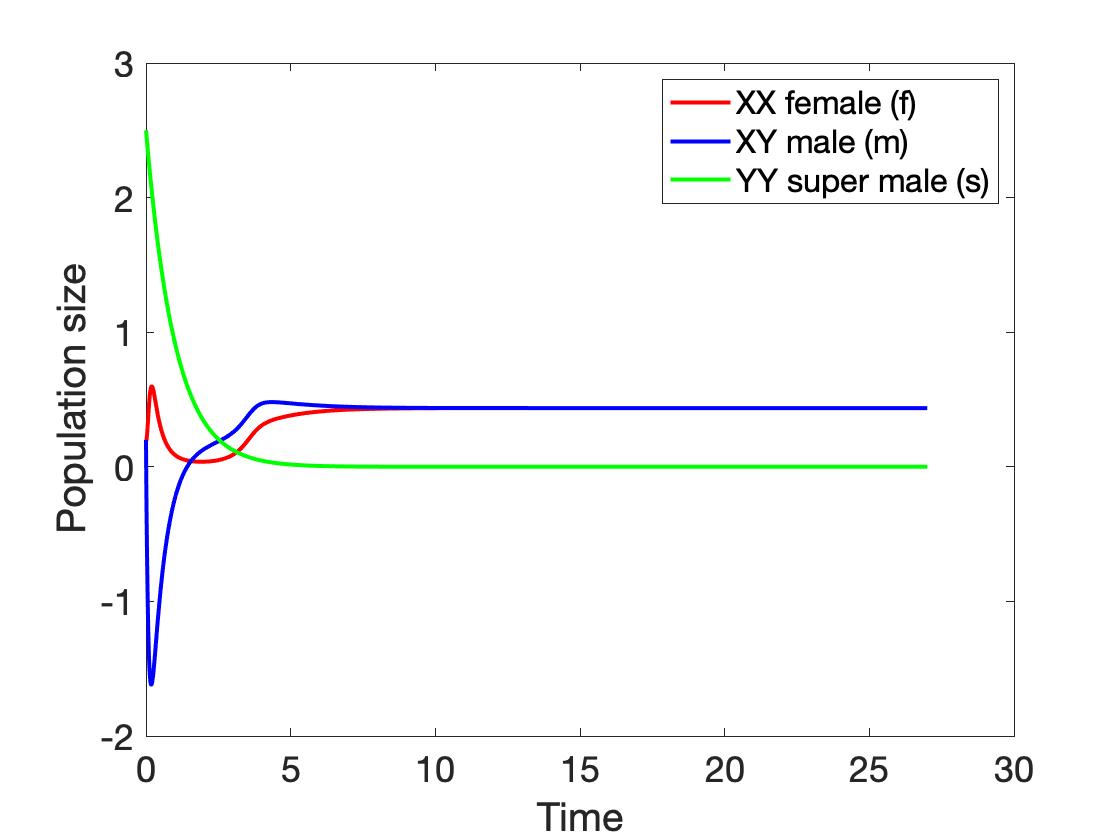}\\
(a) & (b)
\end{tabular}
\begin{tabular}{c}
\includegraphics[scale=0.16]{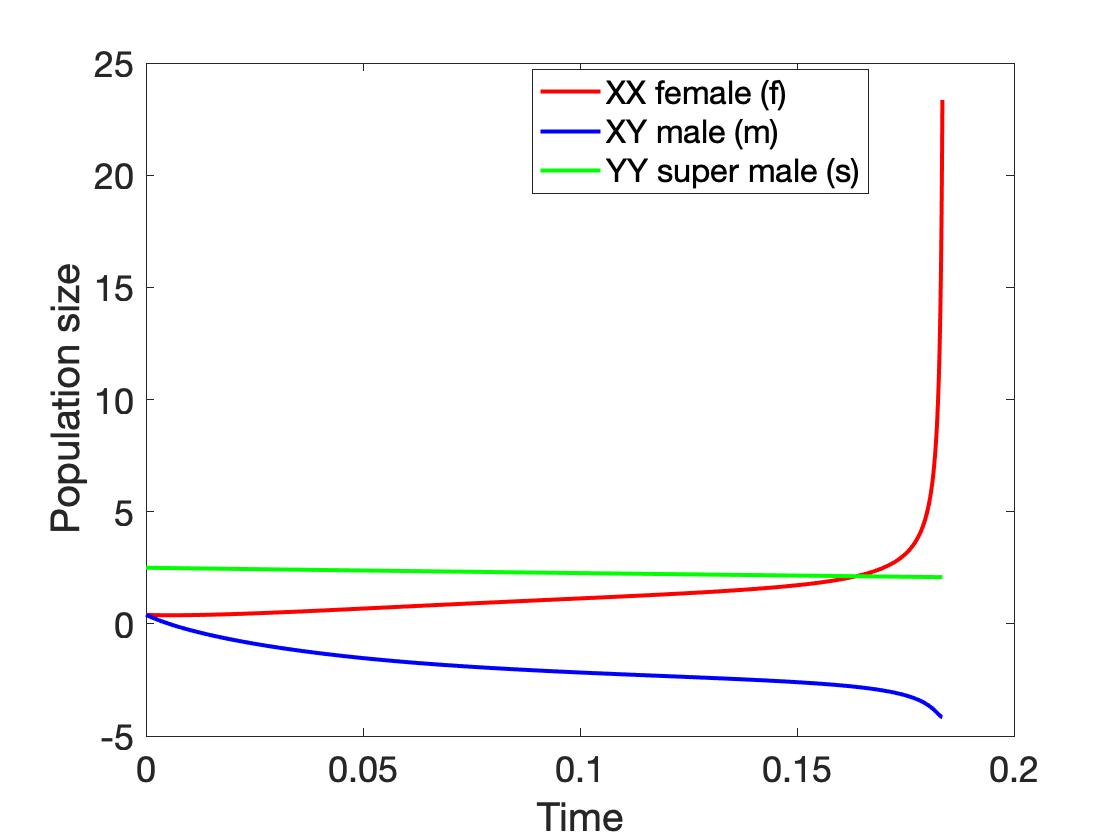} \\
(c)
\end{tabular}
\caption{(a) Positive solutions are shown given the initial conditions of $f(0)=m(0)=.3$ and $s(0) = .1$.  (b) The simulation with $f(0)=m(0)=.3$ and $s(0)=2.5$.  Notice that the male population is clearly negative for an interval.  Hence, $s(0)=2.5 > s^*$. (c) The population densities given initial conditions of $f(0)=m(0)=.4$ and $s(0)=2.5$. The female population is tending towards infinity and we estimate the blow-up time as $T^* \approx 0.18.$}
\label{fig:negMODE}
\end{figure}

In Fig.~\ref{fig:negMODE}(a) we show the populations over time for initial conditions $f(0)=m(0)=.3$ and an initial introduction rate of supermales of $s(0)=.1$ with $\gamma=0.$  It is clear, that the populations remain nonnegative through the computational domain.  However, as we increase the initial supermale population then negative solutions can persist in the male population.  This is shown in Fig.~\ref{fig:negMODE}(b) in the case of $f(0)=m(0)=.3$ and $s(0)=2.5$.  Clearly, $f(0) + m(0) + s(0) = 3.1 > 1,$ which leads to $L<0$ initially.  This causes a large decline in the male population since $\dot{m}(0) \ll 0$.  Subsequently, the male population becomes negative.

In turn, we determine the threshold $s^*$, which may depend on $f(0)$ and $m(0)$, such that $\forall s(0) \geq s^*$ there exists an interval $I \subset (0, \infty)$ for which $\forall t \in I$ we have $m(t)<0$.  In the situation for $f(0) = m(0) = .3$ then $s^*\approx 0.9194$.  In Fig.~\ref{fig:thresholds} we show the critical threshold, $s^*$, for initial populations $f(0) = m(0) \in [.1, .5].$

\begin{figure}[ht]
   \centering
\includegraphics[scale=0.3]{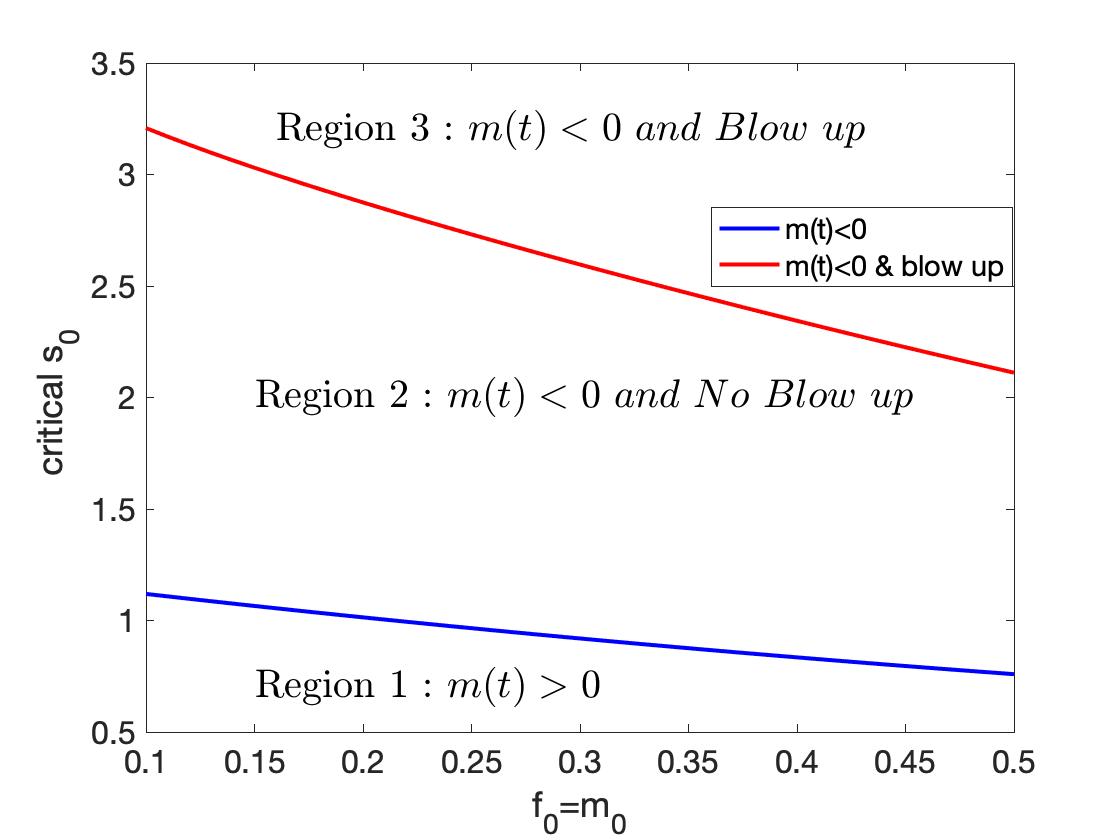}
\caption{Three regions in the phase space are shown.  In Region 1, where $s(0)<s^*$, positive solutions are guaranteed.  In Region 2, where $s^{**}>s(0)\geq s^*$, negative solutions exist but finite time blow up does not occur in either the female or male populations.  In Region 3, for $s(0)\geq s^{**}$ negative solutions exist and blow-up in finite time occurs.}
\label{fig:thresholds}
\end{figure}

If the initial supermale populations is further increased blow-up is possible.  We define $s^{**}$ as the threshold value of initial supermales such that $\forall s(0) \geq s^{**}$ then $\limsup_{t \rightarrow T^{*} < \infty} f \rightarrow + \infty$ for finite time $T^*$, deemed the blow-up time.  To illustrate, let $f(0) = m(0) = .4$ and $s(0) = 2.5$.  Then $f(t)$ blows-up in finite time as shown in Fig. \ref{fig:negMODE}(c).  The threshold $s^{**}$ is documented in Fig. \ref{fig:thresholds}.

Next, consider the situation where $s(0) = 0$.  Let $\gamma^*$ be the critical introduction rate of supermales such that for all $\gamma \geq \gamma^*$ there exists an interval $I\subset (0,\infty)$ such that $m(t)<0$ for all $t \in I$.  Further, by Theorem \ref{thm:2a1} it is known that a critical $\gamma^{**}$ exists such that for all $\gamma \geq \gamma^{**}$ the female or male population will blow up in finite time.  Here, we determine the critical introduction rates for initial conditions $f(0) = m(0) \in  [.1,.5]$.  The results are shown in Fig. \ref{fig:critgammaODE}.

\begin{figure}[ht]
   \centering
\begin{tabular}{cc}
\includegraphics[scale=0.30]{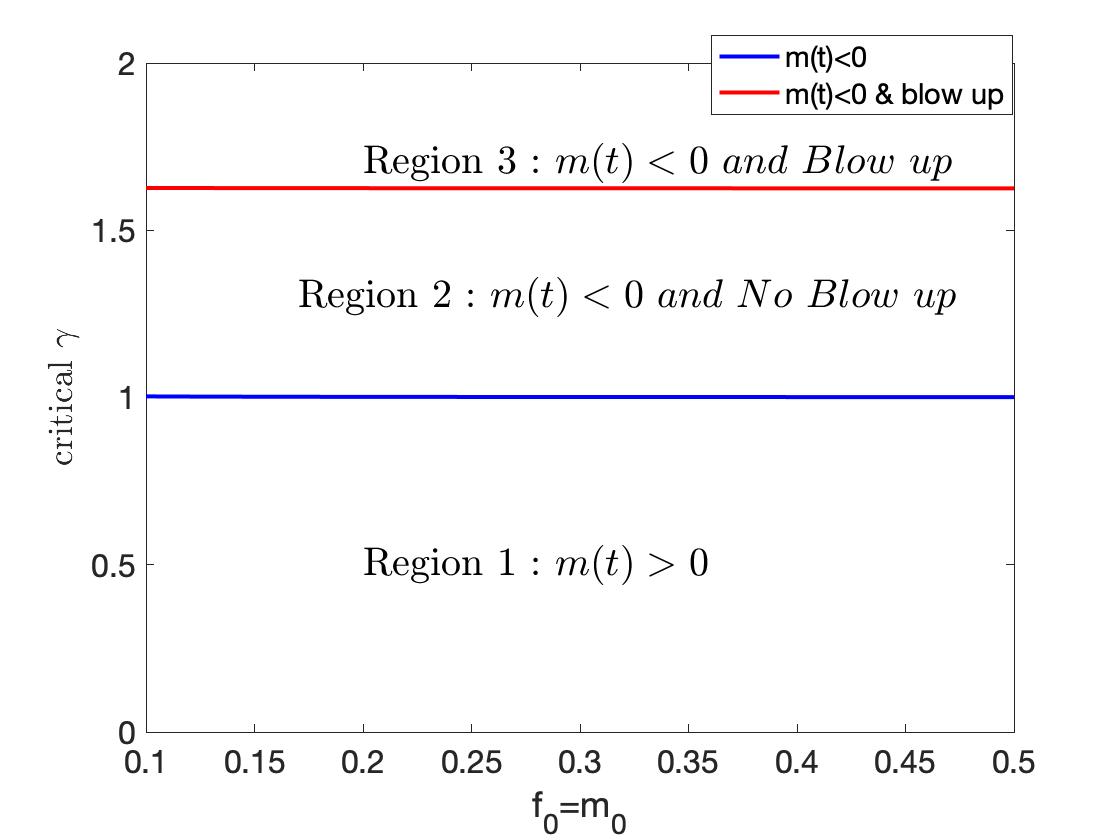}
\end{tabular}
\caption{Three regions in the phase space are shown.  In Region 1, where $\gamma<\gamma^*$, positive solutions are guaranteed.  In Region 2, where $\gamma^{**}>\gamma\geq \gamma^*$, negative solutions exist but finite time blow up does not occur in either the female or male populations.  In Region 3, for $\gamma\geq \gamma^{**}$ negative solutions exist and blow-up in finite time occurs. In each simulation, $s(0) = 0.$  It is evident that both thresholds are independent of the initial condition size.}
\label{fig:critgammaODE}
\end{figure}

As indicated in Section ~\ref{PDE}, blow-up is possible in the partial differential equation model.  Here, we consider the dimensionless spatial-temporal TYC model,
\begin{eqnarray}
\label{eq:nondimFPDE} \frac{\partial f}{\partial t} &=& D \Delta f + rmfL -f \\
\label{eq:nondimMPDE} \frac{\partial m}{\partial t} &=& D \Delta m + rmfL +2rsfL -m\\
\label{eq:nondimSPDE} \frac{\partial s}{\partial t} &=& D \Delta s + \gamma-s
\end{eqnarray}
specified over the scaled spatial domain $\Omega = (0,1)$ and $t\in (0,\infty)$.  Again, homogenous Neumann boundary conditions are assumed.  Clearly, if we assume constant initial conditions, then the solution for $f(x,t)$, $m(x,t)$, and $s(x,t)$ are constant for all $x \in \Omega$ for a fixed value of $t$.  Hence, comparable numerical results are expected in such a case.  In Fig. \ref{fig:PDE}(a) we show a simulation for which negative solutions exist in the male population but finite time blow-up does not occur for $s(x,0) = 2.5$.  Likewise, by increasing the initial condition of supermales to a value of $s(x,0) = 2.75$ we see finite time blow up at $t\approx t=.1899399$.  This is shown in Fig. \ref{fig:PDE}(b).

\begin{figure}[ht]
   \centering
\includegraphics[scale=0.425]{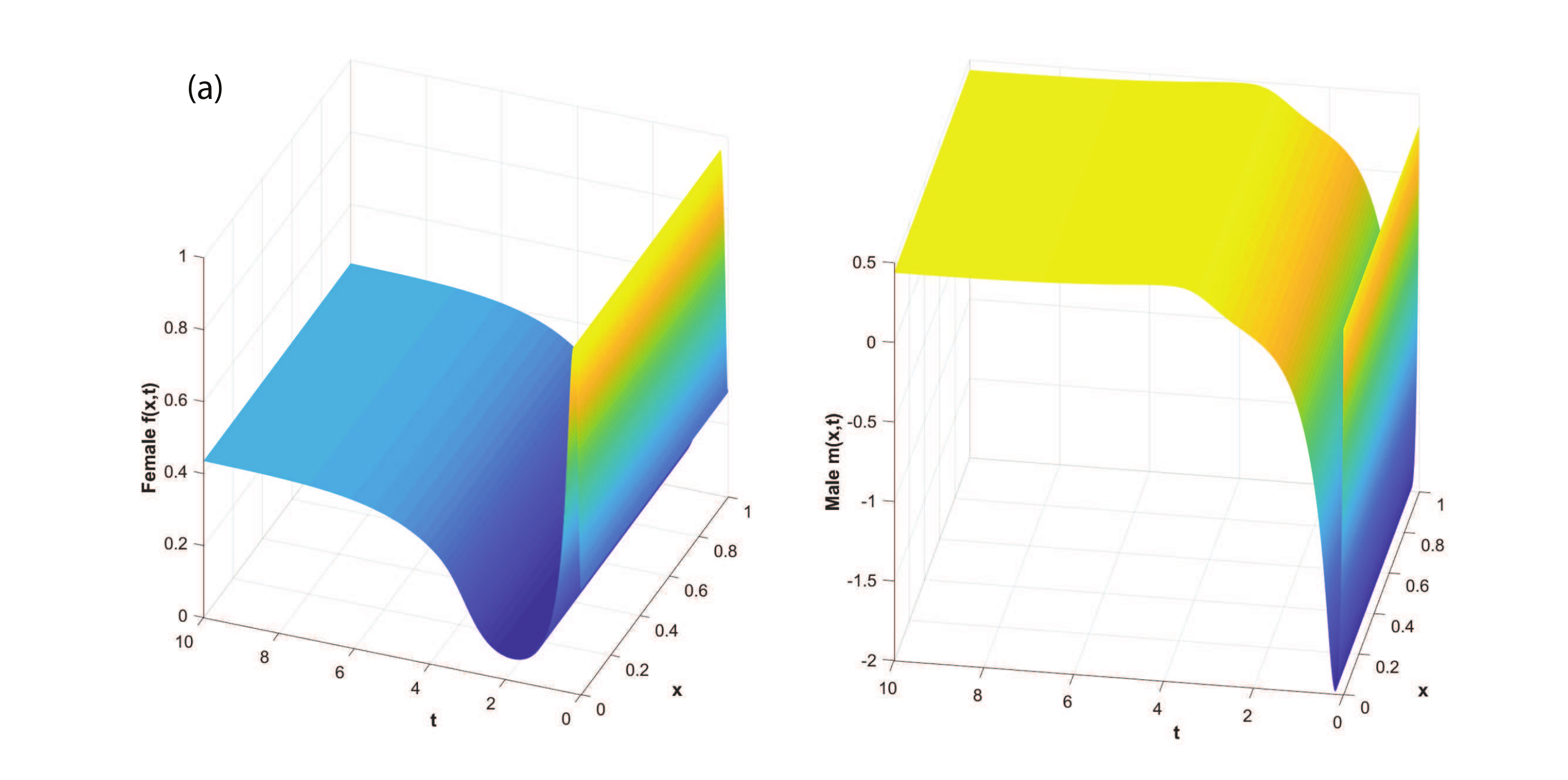}
\includegraphics[scale=0.425]{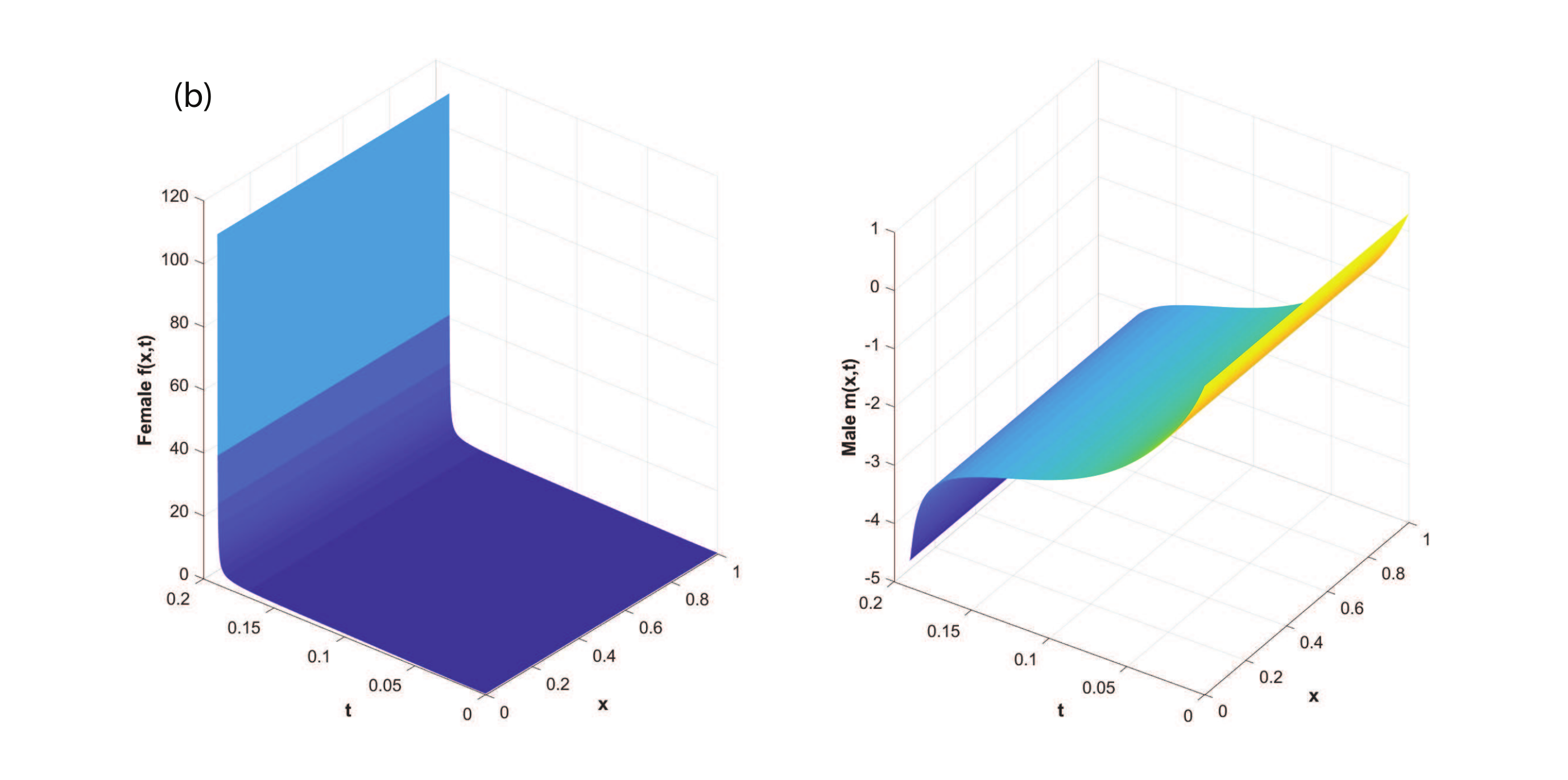}
\caption{(a) Simulation showing negative solutions in the male population for initial conditions $f(x,0)=m(x,0)=.3$ and $s(x,0)=2.5$.  (b) Simulation with $f(x,0)=m(x,0)=.3$ and $s(x,0) = 2.75$.  The increase in the initial amount of supermales results in finite time blow-up in the female population.  Simulations were conducted used Matlab\circledR's pdepe built-in partial differential solver.}
\label{fig:PDE}
\end{figure}

In the case of homogenous Dirichlet boundary conditions it is clear that the boundaries could prevent blow up unless the initial conditions are large enough in norm such that the reaction terms dominate over the diffusive processes.  Indeed, numerical simulations suggest that negativity of solutions and finite time blow-up is possible. Here, we assume homogenous Dirichlet boundary conditions and let the diffusion constant be $.01$.  Assume $f(x,0) = m(x,0) = x(1-x)$ and $s(x,0) = 4 s_{max} x (1-x)$. In Fig. \ref{fig:PDEDIRCH}(a) we show simulations for $s_{max} = 2$ which generates negative solutions in the male population.  In Fig. \ref{fig:PDEDIRCH}(b) finite time blow-up at occurs $t\approx .1901902$ in the female population when $s_{max}$ is increased to $3$.

\begin{figure}[ht]
   \centering
\includegraphics[scale=0.425]{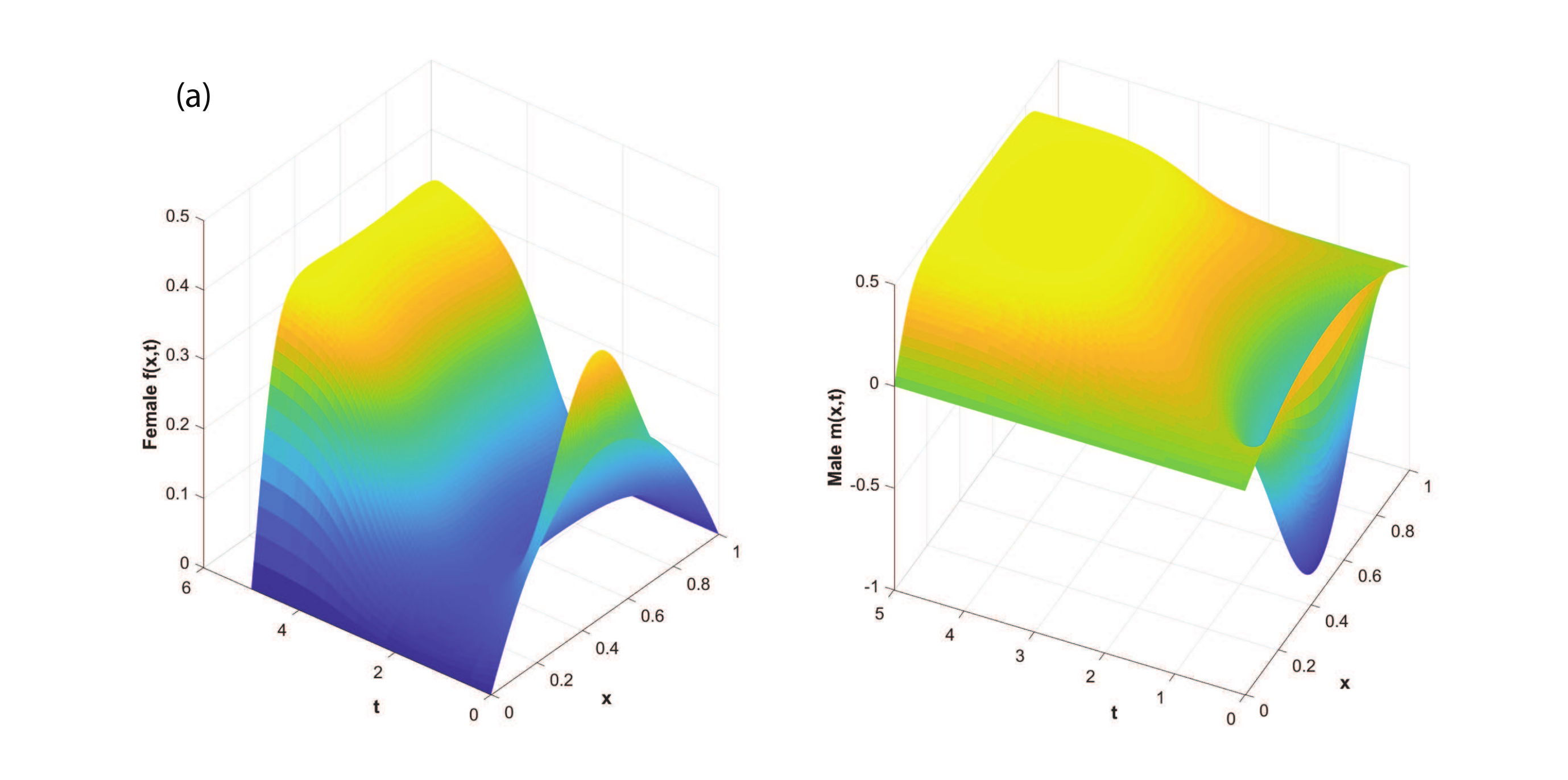}
\includegraphics[scale=0.425]{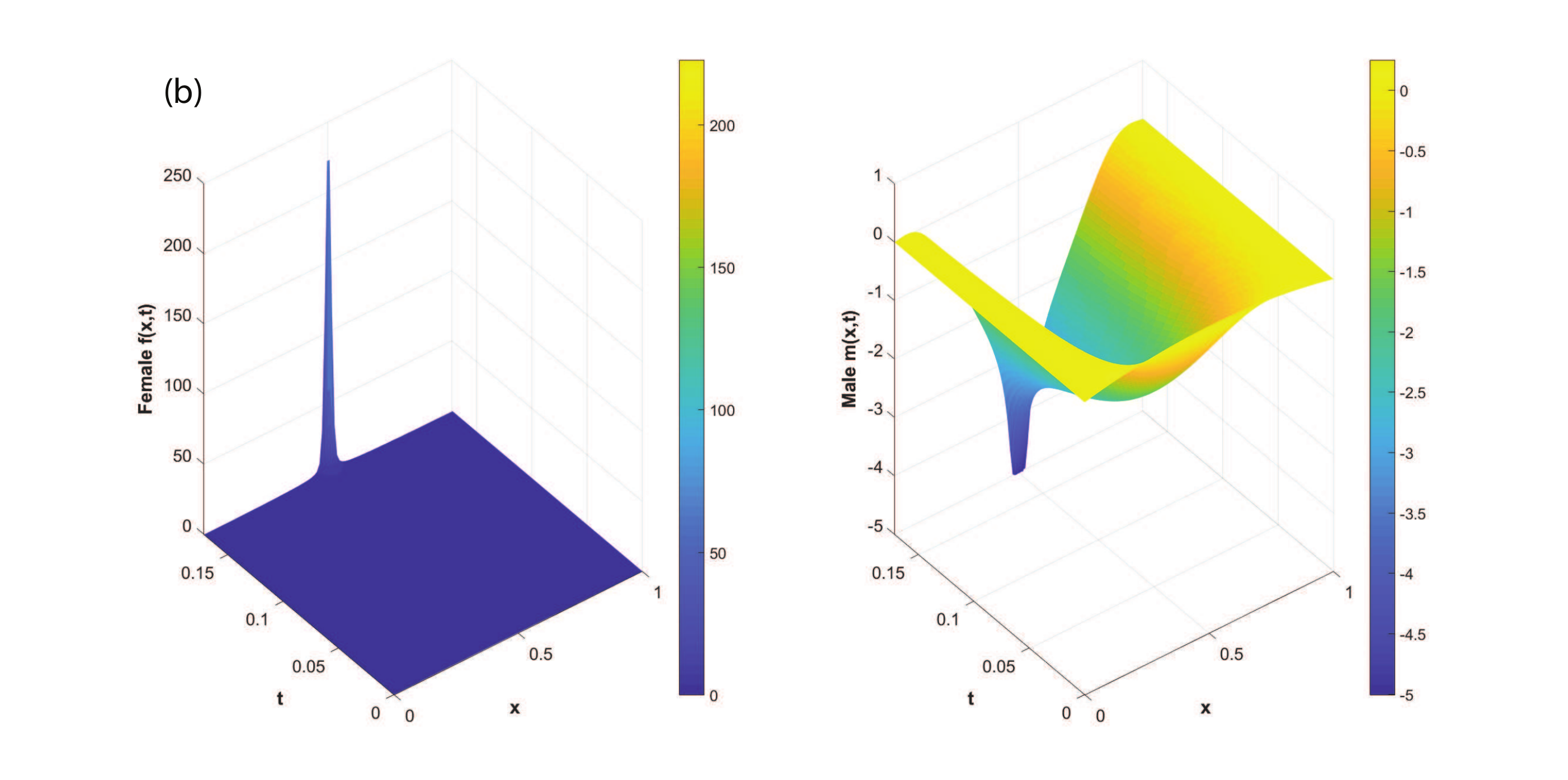}
\caption{(a) Simulation showing negative solutions in the male population for initial conditions $f(x,0)=m(x,0)=x(1-x)$ and $s(x,0)=4s_{max} x(1-x)$, where $s_{max}=2.$  (b) Simulation with the same initial conditions but with $s_{max} = 3$.  The increase in the maximum number supermales results in finite time blow-up in the female population. Each simulation uses $D=.01$. }
\label{fig:PDEDIRCH}
\end{figure}

\section{Discussion and Conclusion}

This paper proves and provides numerous numerical experiments indicating negativity of solutions or finite blow up is possible.  Therefore, caution and discretion should be taken in choosing the parameter regime and initial condition size when utilizing the classical TYC model for predictions of the efficacy of the TYC strategy.

Of course, revisions to the TYC model may remove this inconsistency with physical reality that the current paradigm displays.  A promising revision and modification to the classical TYC model are models that include the Allee effect \cite{BeauOptModTYC2019,SDE2013} and intraspecies competition \cite{BeauOptModTYC2019} for female mates by the males and supermales.  In \cite{BeauOptModTYC2019} a new model to the three species TYC strategy given in the same scaled variables as in \eqref{eq:nondimF}-\eqref{eq:nondimS} is:
\begin{eqnarray}
\label{feq1Mod} \dot{f} &=& r L \left( \frac{f}{a}-1 \right) \left( \frac{m}{m+s} \right) fm - f,\\
\label{meq1Mod} \dot{m} &=& r \frac{Lf}{m+s} \left( \frac{f}{a}-1 \right) \left( m^2 + 2 s^2 \right) - m,\\
\label{seq1Mod} \dot{s} &=& \gamma - s,
\end{eqnarray}
where $a\ll 1$, is the Allee threshold. However, numerical experiments indicate that negativity of solutions remains a possibility.

Assume $\gamma=0$ then for any initial conditions on $f(0)$ and $m(0)$ there exists a critical value of $s^*$ such that male population will become negative for any initial supermale population above or equal to $s^*$. In In Fig.~\ref{fig:thresholdsMOD} we show the threshold in the range of initial values of supermales that leads to negative solutions for $f(0)=m(0) \in [.1, .5]$.  Interestingly, finite time blow-up was not observed in any population.  It is conjectured that the intraspecies competition terms removes this unrealistic dynamic since large populations will attenuate the growth rates of the female and male populations.

\begin{figure}[ht]
   \centering
\includegraphics[scale=0.66]{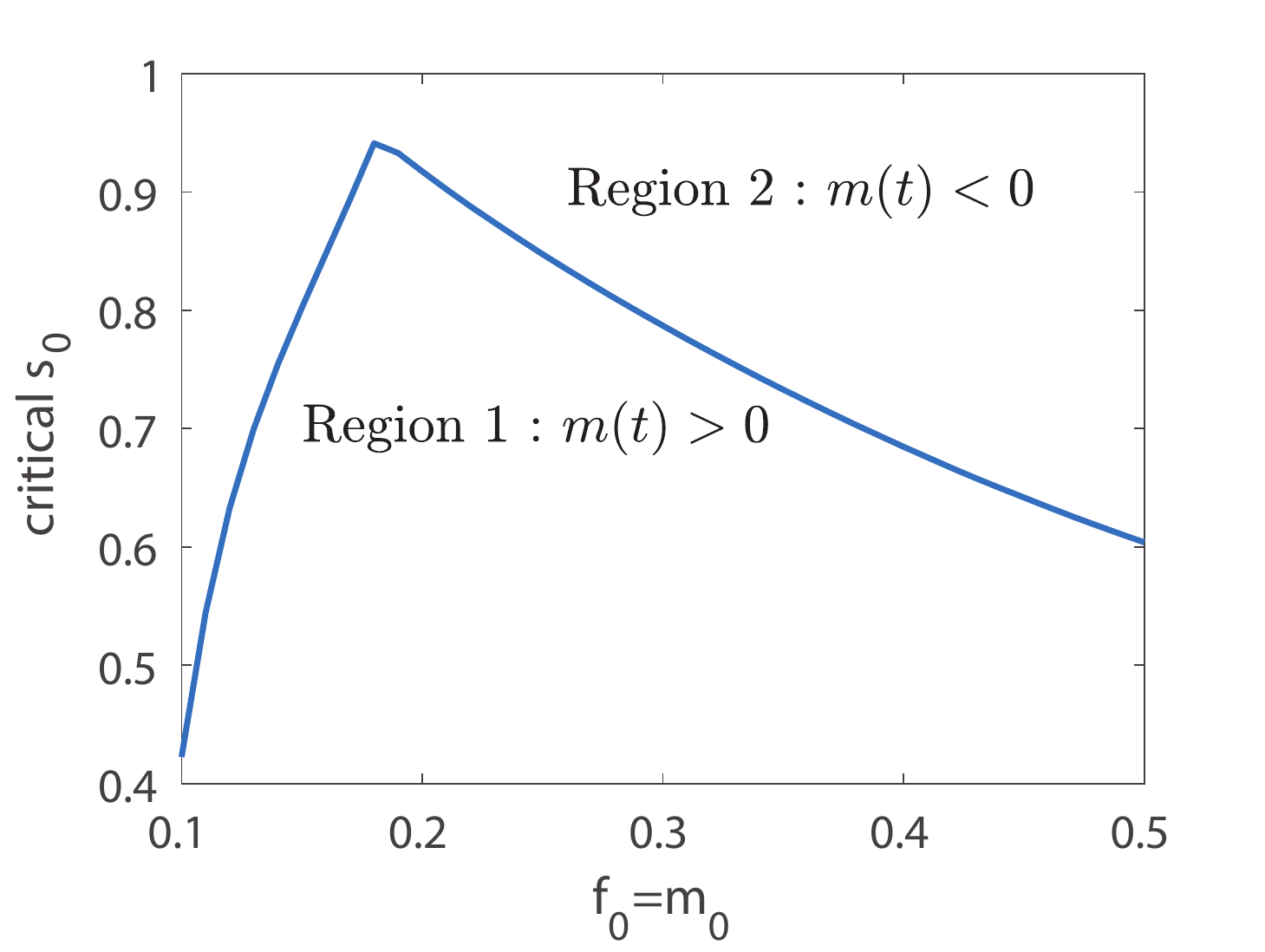}
\caption{Two regions in the phase space are shown.  In Region 1, where $s(0)<s^*$, positive solutions are guaranteed.  In Region 2, where $s(0)\geq s^*$, negative solutions exist. No simulation rendered finite time blow-up in either the female or male populations.}
\label{fig:thresholdsMOD}
\end{figure}

In contrast, the Allee effect affects the differentiability of the threshold curve since at small populations the Allee effect can dominate the dynamics.  To illustrate, we determine the threshold for the model without the Allee effect and see the threshold monotonically decreases as we increase the value of the initial female and male populations.  This is shown in Fig.~\ref{fig:thresholdsMODNOALLEE}.  In fact, the intraspecies competition for mates seems to have little effect on the negative solutions, however, blow-up is not observed.  This suggests, that the inclusion of intraspecies competition is a crucial modeling feature.  However, additional refinement of the TYC model is still required to eliminate the possibility of negative solutions.
\begin{figure}[ht]
   \centering
\includegraphics[scale=0.65]{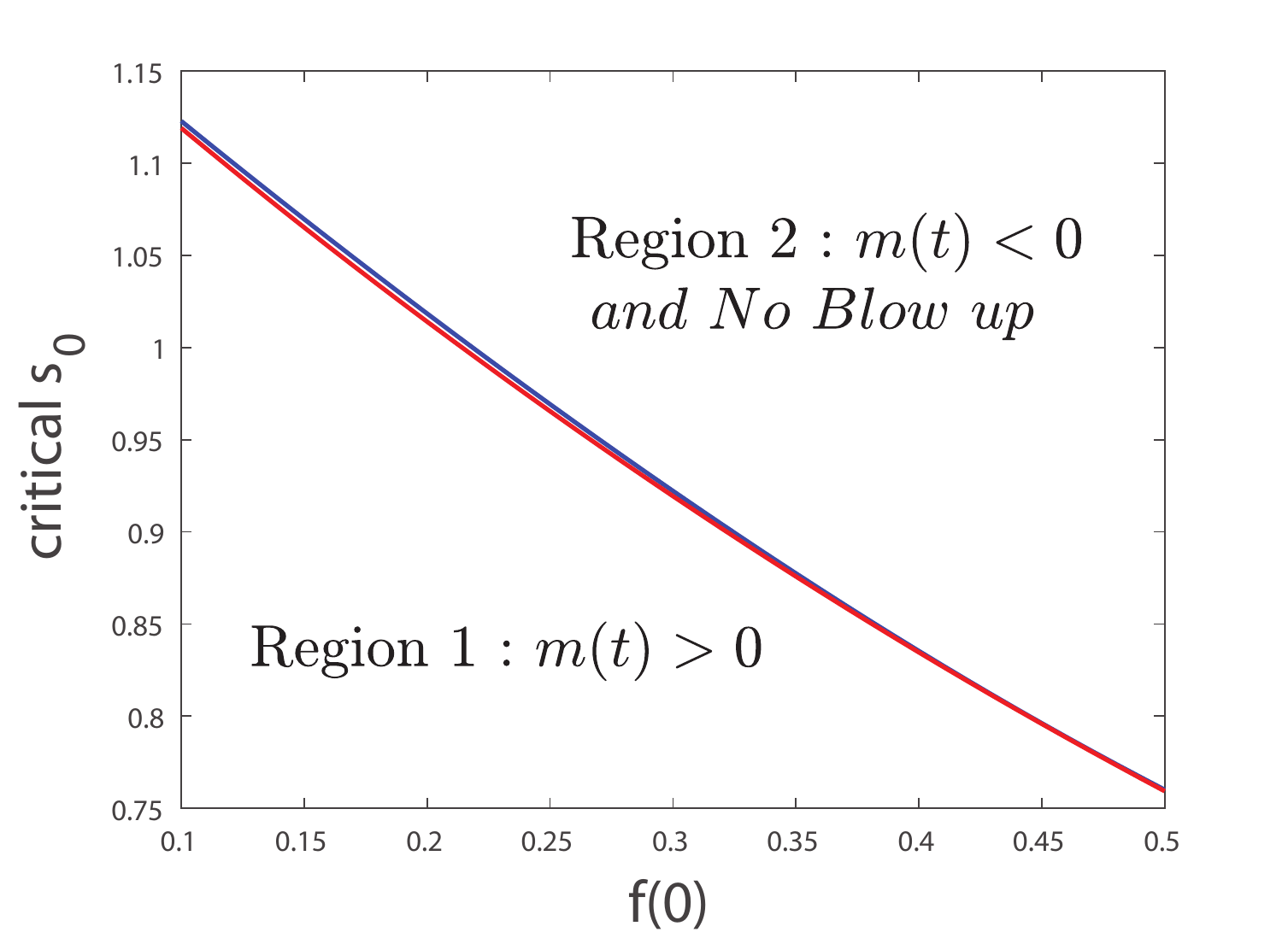}
\caption{Two regions in the phase space are shown.  In Region 1, where $s(0)<s^*$, positive solutions are guaranteed.  In Region 2, where $s(0)\geq s^*$, negative solutions exist. No simulation rendered finite time blow-up in either the female or male populations. The threshold for (red) Eqs. \eqref{eq:nondimF}-\eqref{eq:nondimS}
and (blue) Eqs. \eqref{feq1Mod}-\eqref{seq1Mod} with no Allee effect are given.}
\label{fig:thresholdsMODNOALLEE}
\end{figure}

The cause of the negativity of solutions is a result of the form of the logistic term, $L = 1 - (f+m+s)$.  This term properly stems from consideration of the standard logistic model of population growth, $\dot{P} = \alpha P (1 - P/K).$ Hence, the term $1-P/K$ is the motivation behind $L$ in the model of TYC.  However, in the standard population model when $P>K$ we have $\dot{P}<0$.  In the current situation, $P=f+m+s$ and if $\dot{P}<0$ there is no guarantee that $\dot{f}, \dot{m},$ or $\dot{s}$ are all simultaneously negative nor the populations all positive or comparable size.  Therein, lies the essential modeling flaw.

Alternatives to the logistic term are plentiful.  For instance, one such alternative is to consider a logistic term of the form $\mbox{exp}(1 - (f+m+s))$.  Hence, a populations are still penalized with a dampened growth rate when populations exceed the carrying capacity.  However, preliminary analysis of this type of logistic term have proven difficult for mathematical analysis and moreover can generate stable nontrivial equilibrium solutions such that the total population size asymptotically approaches values above the carrying capacity. Another alternative, is based on the work of mating models (see \cite{Chavez1995,Hadeler1989,Tianran2005} and references therein), where the female/male and female/supermale mating enter are modeled by mating functions $\phi_1(m,f)$ and $\phi_2(s,f)$ such that $\phi_1(0,f)=\phi_1(m,0)=\phi_2(0,f)=\phi_2(s,0)=0$.  Criteria on the mating can then be established to guarantee positivity of solution in the model of the TYC strategy.

In all, this paper indicates it still remains an open problem to determine a complete mathematical model for the TYC strategy that is valid in a full parameter regime and yields realistic solutions.

\section*{Acknowledgements}

ET and RP would like to acknowledge valuable support from the NSF via DMS-1715377 and DMS-1839993. MB, TG and LB would like to acknowledge valuable support from the NSF via DMS-1715044.


\end{document}